\documentclass[12pt, reqno]{amsart}
\usepackage{amssymb,amsmath,amsthm,cancel,hyperref,color, enumerate}
\usepackage[pdftex]{graphicx}
\hypersetup{colorlinks=true,linkcolor=blue,citecolor=magenta}
\newcommand{\pa}[1]{\left( {#1} \right)}

\newcommand{\pf}[2]{\pa{\frac{#1}{#2}}}

\newcommand{\Z}{\mathbb{Z}}
\newcommand{\N}{\mathbb{N}}
\newcommand{\Q}{\mathbb{Q}}

\newcommand{\calO}{\mathcal{O}}

\newtheorem{thm}{Theorem}[section]

\newtheorem{rem}{Remark}

\newcommand{\EGR}{\text{EGR}}
\begin{document}
\newcounter{itemcounter}
\newcounter{itemcounter2}
\title[Elliptic Curves with Everywhere Good Reduction]
{Elliptic Curves with Everywhere Good Reduction}
\author{Amanda Clemm and Sarah Trebat-Leder}
\address{Department of Mathematics, Emory University, Emory, Atlanta, GA 30322}
\email{aclemm@emory.edu, strebat@emory.edu}
\subjclass[2010]{11G05}
\keywords{elliptic curves; quadratic fields; diophantine equations; everywhere good reduction}
\thanks{The second author thanks the NSF for its support.}

\begin{abstract}
We consider the question of which quadratic fields have elliptic curves with everywhere good reduction.  By revisiting work of Setzer, we expand on congruence conditions that determine the real and imaginary quadratic fields with elliptic curves of everywhere good reduction and rational $j$-invariant. Using this, we determine the density of such real and imaginary quadratic fields. If $R(X)$ (respectively $I(X)$) denotes the number of real (respectively imaginary) quadratic fields $K=\mathbb{Q}[\sqrt{m}]$ such that $|\Delta_K| < X$ and for which there exists an elliptic curve $E/K$ with rational $j$-invariant that has everywhere good reduction, then $R(X) \gg \frac{X}{\sqrt{\log(X)}}$. To obtain these estimates we explicitly construct quadratic fields over which we can construct elliptic curves with everywhere good reduction.  The estimates then follow from elementary multiplicative number theory. In addition, we obtain infinite families of real and imaginary quadratic fields such that there are no elliptic curves with everywhere good reduction over these fields. 
\end{abstract}
\maketitle
\noindent

\section{Introduction}
It is a well-known result that over $\mathbb{Q}$ there are no elliptic curves $E$ with everywhere good reduction. However, the same is not true over general number fields. For example, let $K = \mathbb{Q}(\sqrt{29})$ and $a = \frac{5 + \sqrt{29}}{2}$.  Then the elliptic curve $$E: y^2 + xy + a^2 y = x^3$$ has unit discriminant, and hence has everywhere good reduction over $K$.  

This leads to the natural question: Over which number fields do there exist elliptic curves with everywhere good reduction? This question has often been approached by studying $E/K$ with everywhere good reduction which satisfy additional properties, such as those which have a $K$-rational torsion point, admit a global minimal model, or have rational $j$-invariant.  We say that an elliptic curve $E/K$ has $\EGR(K)$ if it has everywhere good reduction over $K$, and that an elliptic curve $E/K$ has $\EGR_{\mathbb{Q}}(K)$ if it additionally has rational $j$-invariant. Similarly, we say a quadratic field has $\EGR$ if there exists a $\EGR(K)$ elliptic curve. 

For many real and imaginary quadratic fields $K$ of small discriminant, explicit examples of elliptic curves $E/K$ with everywhere good reduction can be found in the literature, such as \cite{kida1999reduction} and \cite{ishii1986non}. There are also many known examples of such fields for which there do not exist any elliptic curves $E/K$ with everywhere good reduction; see \cite{kida1999reduction}, \cite{kida1997nonexistence}, \cite{kagawa2000nonexistence} for example.  

Kida \cite{kida1999reduction} showed that if $K$ satisfies certain hypotheses, every $E/K$ with $\EGR$ has a $K$-rational point of order two. This condition led to a series of non-existence results for particular real quadratic fields with small discriminant. In \cite{Setzer:1981vg}, Setzer classified elliptic curves with $\EGR$ over real quadratic number fields with rational $j$-invariant. Kida extended Setzer's approach by giving a more general method suitable for computing elliptic curves with $\EGR$ over certain real quadratic fields with rational or singular $j$-invariants in \cite{kida2000computing}. Comalada \cite{Comalada:1990ur} showed that there exists $E/K$ with $\EGR$, a global minimal model, and a $K$-rational point of order two if and only if one of his sets of diophantine equations has a solution. Ishii supplements this theorem by studying $k-$rational 2 division points in \cite{ishii1986non} to demonstrate specific real quadratic fields without $\EGR$ elliptic curves. Later Kida and Kagawa in \cite{kida1997nonexistence} generalized Ishii's result to obtain non-existence results for $\mathbb{Q}(\sqrt{17)}$, $\mathbb{Q}(\sqrt{73)}$ and $\mathbb{Q}(\sqrt{97)}$. Yu Zhao determined criteria for real quadratic fields to have elliptic curves with $\EGR$ and a non-trivial 3-division point. In \cite{zhao2013elliptic}, he provides a table for all such fields with discriminant less than 10,000. 

For imaginary quadratic fields, Stroeker \cite{Stroeker:1983to} showed that no $E/K$ with $\EGR$ admits a global minimal model.  In \cite{Setzer:1978ub}, Setzer showed that there exist elliptic curves with $\EGR$ and a $K$-rational point of order two if and only if $K = \Q(\sqrt{-m})$ with $m$ satisfying certain congruence conditions. Comalada and Nart provided criteria to determine when elliptic curves have $\EGR$ in \cite{comalada1992modular}. Kida combined this result with a method of computing the Mordell-Weil group in \cite{kida2001good} to prove there are no elliptic curves with $\EGR$ over the fields $\mathbb{Q}(\sqrt{-35}), \mathbb{Q}(\sqrt{-37}), \mathbb{Q}(\sqrt{-51})$ and $\mathbb{Q}(\sqrt{-91})$. 
The following tables shows what it known for $0 \leq \Delta_K \leq 47$.  

\begin {table}[htbp]
\caption {Existence and Nonexistence of an Elliptic Curve with $\EGR(K)$}
\begin{center}
\begin{tabular}{ |c|c|c| } 
 \hline
 \text{Existence} & \text{Non-existence} \\ 
 \hline
 6 & 2 \\
7 & 3 \\
14 & 5 \\
22 & 10  \\
26 & 11   \\
29 & 13   \\
33 & 15 \\
37 & 17 \\
48 & 19 \\
41 & 21 \\
65 & 23 \\
& 30 \\
& 31 \\
& 34 \\
& 35 \\
& 39 \\
& 42 \\
& 43 \\
& 46 \\
& 47 \\
 \hline
\end{tabular}
\end{center}
\end{table}

A combination of the above results gives many methods to prove that a particular quadratic number field has an $\EGR$ elliptic curve. Cremona and Lingham \cite{Cremona:2007vb} described an algorithm for finding all elliptic curves over any number field $K$ with good reduction outside a given set of primes. However, this procedure relies on finding integral points on certain elliptic curves over $K$, which can limit its practical implementation. As a consequence of Setzer's result regarding the classification of elliptic curves over both real and imaginary quadratic number fields with rational $j$-invariant, it is known that there infinitely many quadratic fields which have an $\EGR$ elliptic curve. However, there is no conjectured density result for the proportion of  quadratic fields over which there exist elliptic curves $E$ with everywhere good reduction.

Let $R(X)$ be the number of real quadratic number fields $K$ with discriminant at most $X$ and an $\EGR{_\mathbb{Q}}(K)$ elliptic curve. By revisiting the results of Setzer, we prove the following.  

\begin{thm}\label{rsqrtbound}
With $R(X)$ as above, we have that
$$R(X) \gg \frac{X}{\sqrt{\log(X)}}.$$
\end{thm}

If $I(X)$ is the number of imaginary quadratic number fields $K$ with $|\Delta_K|<X$ and an $\EGR_{\mathbb{Q}}(K)$ elliptic curve, we also obtain the result below. 
\begin{thm}\label{isqrtbound}
With $I(X)$ as above, we have that
$$I(X) \gg \frac{X}{\sqrt{\log(X)}}.$$
\end{thm}

To prove Theorem~\ref{rsqrtbound}, we first show that all real quadratic fields of the form described below in Theorem~\ref{thm237} have $\EGR$, and then count these fields.  

\begin{thm}
\label{thm237}
Let $m = 2q$, where $q = q_1 \cdots q_n \equiv 3 \pmod{8}$ with $q_j \equiv 1, 3 \pmod{8}$ distinct primes.  Then the real quadratic field $K = \Q(\sqrt{m})$ has $\EGR$.  
\end{thm}

\begin{rem}
If $m$ is as described in Theorem~\ref{thm237}, there exists $E/K$ with $\EGR$ and $j(E) = 20^3$.
\end{rem}

Similarly, to prove Theorem~\ref{isqrtbound}, we show all imaginary quadratic fields found below in Theorem~\ref{thm238} have $\EGR$. 

\begin{thm}
\label{thm238}
Let $m = 37q$, where $q = -q_1 \cdots q_n \equiv 1 \pmod{8}$ with $q_j$ distinct primes such that $\pf{q_j}{37} = 1$. Then the imaginary quadratic field $K = \Q(\sqrt{m})$ has $\EGR$. 
\end{thm}

\begin{rem}
If $m$ is as described in Theorem~\ref{thm238}, there exists $E/K$ with $\EGR$ and $j(E) = 16^3$.
\end{rem}

We can achieve results like Theorem~\ref{thm237} and \ref{thm238} for integers other than 2 and 37; these two cases are all is required to prove Theorem~\ref{rsqrtbound} and \ref{isqrtbound}.  

To obtain a density result for $m = qD$, where $D$ is fixed and $q$ varies, we define certain `good' $D$. We say $D$ is good if it is the square free part of $A^3 - 1728$, where $A$ satisfies certain congruence conditions modulo powers of 2 and 3. These congruence conditions will be described explicitly in Section 2. If $D$ is good, then $K = \mathbb{Q}\sqrt{Dq}$ has $\EGR$ whenever $D$ and $q$ satisfy certain explicit conditions, see Section~\ref{Section2}.  Define 
 \begin{displaymath}
\epsilon_D = \left\{
\begin{array}{l l}
1 & D \equiv 1 \pmod{4} \\
-1 & \text{otherwise} \\
\end{array}
\right.
\end{displaymath}
When $\text{sign}(D) = -\epsilon_D$, we get real quadratic fields $\Q(\sqrt{qD})$, and when $\text{sign}(D) = \epsilon_D$, we get imaginary quadratic fields.  

Using this, we show that $R_D(X)$, the number of $q \leq X$ such that $Q(\sqrt{Dq})$ is a real $\EGR$ quadratic number field, satisfies the following lower bound: 

\begin{thm}
\label{genboundr}
Let $D$ be good with $r$ distinct prime factors and $R_D(X)$, the number of $\EGR$ real quadratic number fields $Q(\sqrt{Dq})$ with $q \leq X$. 
Assume that $\text{sign}(D) = -\epsilon_D$.  Then $$R_D(X) \gg \frac{X}{\log^{1 - 1/2^r} X}.$$
\end{thm}

We obtain a similar result to show that $I_D(X)$, the number of $\EGR$ imaginary quadratic number fields $Q(\sqrt{Dq})$ satisfies the following lower bound. 

\begin{thm}
\label{genboundim}
Let $D$ be good with $r$ distinct prime factors and $I_D(X)$, the number of $\EGR$ imaginary quadratic number fields $Q(\sqrt{Dq})$ with $q \leq X$. Assume that $\text{sign}(D) = \epsilon_D$.  Then $$I_D(X) \gg \frac{X}{\log^{1 - 1/2^r} X}.$$
\end{thm}

\begin{rem}
While we have only looked at curves with rational $j$-invariant, Noam Elkies' computations \cite{Elkies:2014:Online} suggest that very few $E/K$ with $\EGR$ have $j(E) \not \in \Q$ and unit discriminant. Therefore, the theorem below, which to the best of our knowledge has not previously appeared in the literature, suggests that most fields of the form $K = \Q(\sqrt{\pm p})$ for primes $p \equiv 3 \pmod{8}$ are not $\EGR$. This is consistent with Elkies' data. 
\end{rem} 

Using this approach we were also able to determine nonexistence of $\EGR$ quadratic fields. 

\begin{thm}
\label{nonexistence}
Let $p \equiv 3 \pmod{8}$ be prime. 
\begin{enumerate}
\item Let $K = \Q(\sqrt{p})$.  Then there are no $E/K$ with $\EGR$ and $j(E) \in \Q$.  

\item let $K = \Q(\sqrt{-p})$.  Then there are no $E/K$ with $\EGR$ and $j(E) \in \Q$.
\end{enumerate}
\end{thm}

\begin{rem}
In \cite{kagawa2000nonexistence}, Kagawa showed that if $p$ is a prime number such that $p \equiv 3 (4)$ and $p \neq 3, 11$, then there are no elliptic curves with $\EGR$ over $K=\mathbb{Q}(\sqrt{3p})$ whose discriminant is a cube in $K$. Since  all $\EGR_{\mathbb{Q}}(K)$ curves have cubic discriminant as shown in Setzer \cite{Setzer:1981vg}, this gives a result similar to Theorem 1.7.\end{rem}
  

In Section~\ref{Section2}, we describe conditions arising from Setzer to define when we have $\EGR$ quadratic fields. In Section~\ref{Section3}, we use these conditions to find a lower bound based on an example of Serre. In Section~\ref{Section4}, we will give examples of $\EGR$ real quadratic fields and $\EGR$ imaginary quadratic fields. 

\section{Constructing $\EGR$ Quadratic Fields}
\label{Section2}
In \cite{Setzer:1981vg}, given a rational $j$-invariant, Setzer determines whether there exists an elliptic curve and number field over which this curve has everywhere good reduction. Following his notation, we make the following definitions. Let $\mathcal{R}$ be the following set:
\[
\mathcal{R}= \{A \in \mathbb{Z} : 2|A\Rightarrow 16|A \text{ or } 16|A-4,\text{ and } 3|A \Rightarrow 27|A-12\}.  
\]
Then $D$ is good if it is in the following set:
$$\{D: Dt^2 = A^3 - 1728, D \text{ square-free}, A \in \mathcal{R}, t \in \Z\}.$$
Given $D$ good, we define $\epsilon_D$ as follows: 

 \begin{displaymath}
\epsilon_D = \left\{
\begin{array}{l l}
1 & D \equiv 1 \pmod{4} \\
-1 & \text{otherwise} \\
\end{array}
\right.
\end{displaymath}

\begin{rem}We note that $\pm 1$ are not good, as the elliptic curves $Y^2 = X^3 - 1728, -Y^2 = X^3 - 1728$ have no integral points with $Y \neq 0$.  
\end{rem}

By Setzer \cite{Setzer:1981vg}, the only candidates for elliptic curves $E$ with $\EGR_{\mathbb{Q}}(K)$ over a quadratic field $K$ have $j(E)=A^3$ with $A \in \mathcal{R}$.  

\begin{thm}[\cite{Setzer:1981vg}]
\label{Setzer}
Let $K = \Q(\sqrt{m})$ be a quadratic field.  Then there exists an elliptic curve $E/K$ with $\EGR$ and rational $j$-invariant if and only if the following conditions are satisfied for some good $D \mid \Delta_K$.  
\begin{enumerate}
\item $\epsilon_D D$ is a rational norm from $K$. 
\item If $D \equiv \pm 3 \pmod{8}$, then $m \equiv 1 \pmod{4}$.
\item If $D$ is even then $m \equiv 4 + D \pmod{16}$.
\end{enumerate}
\end{thm}

To prove the theorem, Setzer shows that given a pair $(m, D)$ satisfying the conditions of the theorem, there exists $u \in K^\times$ such that $$E_{u, A}: y^2 = x^3 - 3A(A^3 - 1728)u^2 x - 2(A^3 - 1728)^2 u^3$$ has $j$-invariant $A^3$ and $\EGR$ over $K$. 

\begin{rem}
We correct a mistake in Condition (2) of this theorem as written in \cite{Setzer:1981vg}. 

We note that if $u \equiv v \pmod{4 \calO_K}$ and $m \equiv 2, 3, \pmod{4}$, then we must have that $N(u) \equiv N(v) \pmod{8}$.  However, if $m \equiv 1 \pmod{4}$, we only know that $N(u) \equiv N(v) \pmod{4}$.  Moreover, we can pick $w \in 4\calO_K$ such that $N(u + w) \equiv N(u) + 4 \pmod{8}$.  

Condition (2) as written in Setzer's paper states that if $D \equiv \pm 3 \pmod{8}$, then $m \equiv 5 \pmod{8}$.  $D \equiv \pm 3 \pmod{8}$ implies that a certain element $u \in \calO_K$ has $N(u) \equiv 5 \pmod{8}$.  But for the curve to have good reduction at primes dividing $2$, it is necessary that $u$ is congruent to a square modulo $4 \calO_K$.  For $m \equiv 2, 3 \pmod{4}$ this is not possible, as no squares can have norm equivalent to $5$ modulo $8$. However, if $m \equiv 1 \pmod{4}$, the condition that $N(u) \equiv 5 \pmod{8}$ is not an obstacle, as $u$ is congruent modulo $4 \calO_K$ to elements of norm $1$ modulo $8$.  Setzer mistakenly assumes that this can only happen when $m \equiv 5 \pmod{8}$.  
\end{rem}

In proving that fields do and do not have elliptic curves with $\EGR$ and rational $j$-invariant, the following equivalent version of Setzer's theorem will be useful.  

\begin{thm}\label{congrucond}
Fix $D$ good, and $m=qD$ square-free.   $K = \Q(\sqrt{m})$ has $\EGR$ if and only if the following conditions are satisfied:
\begin{enumerate}[(a)]
\item $(-\epsilon_D q / p_i)=1$ for all primes $p_i$ dividing $D$;
\item $(\epsilon_D D / q_j)=1$ for all primes $q_j$ dividing $q$; 
\item $m >0$ if $\epsilon_D D <0$;
\item If $D \equiv \pm 3\pmod 8$ then $q \equiv D \pmod 4$;
\item If $D$ is even then $q \equiv D+1 \pmod 8$.
\end{enumerate}
\end{thm}

\begin{proof}[Proof of Theorem \ref{congrucond}]
We need to show that the conditions in Theorem~\ref{Setzer} are equivalent to those in Theorem~\ref{congrucond}.  

Assume that $K = \Q(\sqrt{m})$ where $m$ is square-free.  

Clearly if $m=qD$, $D$ divides $\Delta_K$. We need to show that if $D \mid \Delta_K$ then $D \mid m$.  This is trivial for $m \equiv 1 \pmod{4}$, as then $\Delta_K = m$. If $m \equiv 3 \pmod{4}$, then $D$ cannot be even because of (3), so $D \mid m$. If $m \equiv 2 \pmod{4}$, then $D$ must be square-free, so $D \mid m$.  

Now, $\epsilon_D D$ is a rational norm from $K$ if and only if  there exists a rational solution to $\epsilon_D D = a^2 - b^2 D q$.  Since $D \mid a$, the above is equivalent to the existence of a rational solution to $\epsilon_D = D (a')^2 - b^2 q$. Condition $(c)$ insures that $\epsilon_D$, $D$ and $-q$ are not all positive nor all negative. Then by Legendre's Theorem, \cite{ireland1982classical}, for $p_i$ a prime not dividing $m$,  there are always local solutions to $\epsilon_D D = a^2 - b^2 D q$ if and only if conditions (a) - (c) are satisfied. 

Conditions (d) and (e) are directly equivalent to (2) and (3).  
\end{proof}

To prove Theorem 1.1, the lower bound for $R_D(X)$ and Theorem 1.2, the lower bound for $I_D(X)$, we require Theorem 1.3 (which considers the case $D=2$) and Theorem 1.4 (which considers the case $D=37$). Below, we prove both those theorems using the result above.  

\begin{proof}[Proof of Theorem~\ref{thm237}]
Let $A = 20 \in \mathcal{R}$.  This shows that $D = 2$ is good.  For $m = 2 q$ with $q = q_1 \cdots q_n \equiv 3 \pmod{8}$ and $q_j \equiv 1, 3 \pmod{8}$ distinct primes, all of the conditions in Theorem~\ref{congrucond} are satisfied, and so $K = \Q(\sqrt{m})$ has $\EGR$.  \\
\end{proof}

\begin{proof}[Proof of Theorem~\ref{thm238}]
Let $A = 16 \in \mathcal{R}$.  This that shows that $D = 37$ is good.  For $m = 37q$ with $q = -q_1 \cdots q_n \equiv 1 \pmod{8}$ and $q_j$ distinct primes such that $\pf{q_j}{37} = 1$, all of the conditions in Theorem~\ref{congrucond} are satisfied, and so $K = \mathbb{Q}(\sqrt{m})$ has $\EGR$.  
\end{proof}

We also can use Theorem~\ref{congrucond} to prove nonexistence results about $\EGR$ quadratic fields. 

\begin{proof}[Proof of Theorem~\ref{nonexistence}]
Let $p \equiv 3 \pmod{8}$ be prime.  

To show that there are no $E/\mathbb{Q}(\sqrt{p})$ with $\EGR$ and rational $j$-invariant, we must show that neither of the pairs $(D, q) =  (\pm p, \pm 1)$ satisfy the conditions of Theorem~\ref{congrucond}.   We note that since $p=D \equiv \pm 3 \pmod{8}$, condition (d) implies that $q \equiv 5D \equiv \pm 1 \pmod{8}$, which is a contradiction. 

Similarly, to show that there are no $\EGR_{\mathbb{Q}}(\mathbb{Q}(\sqrt{-p})$, we have to show that neither of the pairs $(D, q) = (\pm p, \pm 1)$ satisfy the conditions of the theorem.  We note that in both cases, condition (a) implies that $\pf{-1}{p} = 1$, which is a contradiction. 
\end{proof}

\section{Finding Lower Bounds}
\label{Section3}
To prove the lower bounds, we use an example of Serre \cite{Serre72} as a reference.  

We define a set $E \subset \N_{>0}$ to be multiplicative if for all pairs $n_1, n_2$ relatively prime, we have that $n_1 n_2 \in E$ if and only if  $n_1 \in E$ or $n_2 \in E$.  Given a set $E$, let $P(E)$ be the set of primes $p$ in $E$.  Let $\bar{E} := \N_{>0} - E$, and $\bar{E}(X) := \{m \in \bar{E}, m \leq X\}$.  

\begin{thm}[\cite{Serre72}]
\label{Serre}
Suppose that $E$ is multiplicative and $P(E)$ has Chebotarev set $0 < \alpha < 1$.  Then $$\bar{E}(X) \sim cX/\log^\alpha X$$ for some $c > 0$.  
\end{thm}

We will use the theorem above to prove Theorem~\ref{genboundr} and Theorem~\ref{genboundim}.  As shown in Section 2, the special cases with $D = 2, 37$ will then imply Theorem~\ref{rsqrtbound} and \ref{isqrtbound}.  

\begin{proof}[Proof of Theorem~\ref{genboundr}]
Let $D$ be good. 

If $D$ is odd, let $D = \pm 1p_1 \cdots p_r$ be its prime factorization. Then we define $$\bar{E} := \{q_1^{a_1} \cdots q_n^{a_n}: \text{$q_j$ is prime, $a_j \geq 0$, $\pf{q_j}{p_i} = 1$}\}.$$

If $D$ is even, let $D = \pm 2 p_1 \cdots p_{r - 1}$ be its prime factorization.  Define
\begin{displaymath}
\delta = \left\{
\begin{array}{l l}
1 & \text{if $D/2 \equiv 1 \pmod{4}$} \\
-1 & \text{if $D/2 \equiv -1 \pmod{4}$} \\
\end{array}
\right.
\end{displaymath}

Let $$\bar{E} := \{q_1^{a_1} \cdots q_n^{a_n}: \text{$q_j$ is prime, $a_j \geq 0$, $\pf{q_j}{p_i} = 1$, $\pf{-2\delta }{q_j} = 1$}\}.$$

In both cases, $E$, the set such that $\bar{E} := \N_{>0} - E$, is multiplicative, and $P(E)$ has Chebotarev set $\alpha = 1 - 1/2^r$.  Therefore, by Theorem~\ref{Serre}, we have  
$$\bar{E}(X) \sim c X/\log^\alpha X.$$ Now, we have to relate $\bar{E}(X)$ to $R_D(X)$ and $I_D(X)$.  We will do this for three cases, when $D \equiv 3, 1, 2 \pmod{8}$, as the others work similarly.  

\begin{enumerate}

\item Assume that $D \equiv 3 \pmod{8}$, $D > 0$. Consider the set 
$$E'(X) := \{Dq: q \in \bar{E}(X), q \text{ square-free, and $q \equiv 7 \pmod{8}$}\}.$$ 
We first show that $$E'\subset R_D(X).$$ To do this we must show $E'$ satisfies conditions $(a)-(e)$ in Theorem~\ref{congrucond}. Condition $(a)$ is satisfied as:
$$\pf{-\epsilon_D q}{p_i} = \pf{q}{p_i} = \prod_j\pf{q_j}{p_i} = 1.$$
Condition $(b)$ follows from:
$$\pf{\epsilon_D D}{q_j} = \pf{-D}{q_j} = \pf{q_j}{D} = \prod_i \pf{q_j}{p_i} = 1.$$    
Both condition $(c)$ and $(d)$ follow directly from the definition of $E'$.  
Since a positive proportion of elements of $\bar{E}(X)$ are also in $E'(X)$ and $E'\subset R_D(X)$, this tells us that a positive proportion of $q \in \bar{E}(X)$ are such that $Dq \in R_D(X)$. Therefore, $$R_D(X) \gg \frac{X}{\log^\alpha(X)}.$$

\item Assume that $D \equiv 1 \pmod{8}, D < 0$.  We have that $$E'(X) := \{Dq: -q \in \bar{E}(X), q \text{ square-free}\} \subset R_D(X),$$ as $$\pf{-\epsilon_D q}{p_i} = \pf{-q}{p_i} = \prod_j \pf{q_j}{p_i} = 1$$ and $$\pf{\epsilon_D D}{q_j} = \pf{D}{q_j} = \pf{q_j}{-D} = \prod_i \pf{q_j}{p_i} = 1.$$  As a positive proportion of elements of $\bar{E}(X)$ are also in $E'(X)$, this tells us that a positive proportion of $-q \in \bar{E}(X)$ are such that $Dq \in R_D(X)$. Therefore, $$R_D(X) \gg \frac{X}{\log^\alpha(X)}.$$

\item Assume that $D \equiv 2 \pmod{8}, D < 0$.  We have that $$E'(X) := \{Dq: q \in \bar{E}(X), q \text{ square-free, and } q \equiv 3 \pmod{8}\} \subset I_D(X),$$ as $$\pf{-\epsilon_D q}{p_i} = \pf{q}{p_i} = \prod_j \pf{q_j}{p_i} = 1$$ and $$\pf{\epsilon_D D}{q_j} = \pf{-2}{q_j}\pf{-(-D)}{q_j} = \pf{q_j}{-D} = \prod_i \pf{q_j}{p_i} = 1.$$  As a positive proportion of elements of $\bar{E}(X)$ are also in $E'(X)$ (note that $\pf{-2}{q_j} = 1$ implies that $q_j \equiv 1, 3 \pmod{8}$), this tells us that a positive proportion of $q \in \bar{E}(X)$ are such that $Dq \in I_D(X)$.   Therefore, $$R_D(X) \gg \frac{X}{\log^\alpha(X)}.$$

Similar calculations can be done for $m<0$, resulting in $$I_D(X) \gg \frac{X}{\log^\alpha(X)}.$$

\end{enumerate}

\end{proof}

\begin{proof}[Proof of Theorem \ref{rsqrtbound}]
The theorem follows immediately from Theorem~\ref{genboundr} and Theorem~\ref{thm237}. Theorem 1.3 shows $D=2$ is good with $r=1$ distinct factors and the real quadratic field $K=\mathbb{Q}(\sqrt{qD})$ has 
$\EGR$. If $R(X)$ is the number of these fields, Theorem 1.5 shows
$$R(X) \gg \frac{X}{\sqrt{\log(X)}}.$$
\end{proof}

\begin{proof}[Proof of Theorem \ref{rsqrtbound}]
The theorem follows immediately from Theorem~\ref{genboundim} and Theorem~\ref{thm238}. Theorem 1.4 shows $D=37$ is good with $r=1$ distinct factors and the imaginary quadratic field $K=\mathbb{Q}(\sqrt{qD})$ has 
$\EGR$. If $I(X)$ is the number of these fields, Theorem 1.6 shows
$$I(X) \gg \frac{X}{\sqrt{\log(X)}}.$$
\end{proof}

\section{Examples}
\label{Section4}
In this section, we explain how to find elliptic curves with $\EGR$ when the conditions of Theorem~\ref{congrucond} are satisfied, and give examples of elliptic curves with $\EGR$. The results in this section are based on Setzer's construction in \ref{Setzer}. 

We start with a quadratic field $K = \mathbb{Q}(\sqrt{m})$ and a factorization $m = Dq$ with $D$ good which satisfies the conditions of Theorem~\ref{congrucond}.   We want to find $u$ such that 
$$E_{u, A}: y^2 = x^3 - 3A(A^3 - 1728)u^2 x - 2(A^3 - 1728)^2 u^3$$ has $\EGR_{\mathbb{Q}}(K)$.  Let $\alpha \in K$ have norm $\epsilon_D D$, and pick $n$ odd such that $\beta := n \alpha = a + b \sqrt{m} \in \calO_K$.  Let $A \in \mathcal{R}$ be such that $D$ is the square-free part of $A^3 - 1728$.  Define $d_1, d_2$ such that $3^2(A^3 - 1728) = Dd_1^2d_2^4$ with $d_1$ square-free. If $m \equiv 1, 2 \pmod{4}$, then one of $u = \pm \beta d_1$ works. If $m \equiv 3 \pmod{4}$, then either $u = \pm \beta d_1$ both work or $u = \pm \beta d_1 \rho$ both work, where $\rho = \frac{1}{2}(m + 1) + \sqrt{m}$. 

The table below has some examples.

$$
\begin{array}{| c | c | c | c | c | c |}
\hline
A & D & d_1 & q & \alpha & u \\
\hline
20 & 2 & 42 & 3 & 2 + \sqrt{6} & -d_1 \alpha = -84 - 42 \sqrt{6}  \\
\hline
-15 & -7 & 1 & -11& 35 + 4 \sqrt{77}& -d_1 \alpha = -35 - 4 \sqrt{77}\\
\hline
-32 & -11 & 42 & -15 & 77 + 6 \sqrt{165} & d_1 \alpha = 3234 + 252 \sqrt{165}\\
\hline
-32 & -11 & 42 & -3 & 11 + 2 \sqrt{33} & -d_1 \alpha = -462 - 84 \sqrt{33} \\
\hline
39  & 79 &1 &5 &79 + 4 \sqrt{395} & \pm d_1 \alpha \rho = \pm(17222 + 871 \sqrt{395}) \\
\hline
16 & 37 & 6 & -7 & 37 + 6 \sqrt{-259} & \pm d_1\alpha = \pm(222 +36 \sqrt{-259}) \\
\hline

\end{array}
$$

\bibliographystyle{plain}
\bibliography{refs.bib}
\end{document}